\documentclass{article}

\usepackage{amsmath}
\usepackage{amssymb}
\usepackage{amscd}
\usepackage{amsthm}
\usepackage{eufrak}
\usepackage[mathscr]{eucal}
\usepackage[all]{xy}
\usepackage{multicol}

\newtheorem{theorem}{Theorem}[section]
\newtheorem{lemma}[theorem]{Lemma}
\newtheorem{proposition}[theorem]{Proposition}

\theoremstyle{definition}
\newtheorem{definition}[theorem]{Definition}

\theoremstyle{remark}
\newtheorem{remark}[theorem]{Remark}

\newtheorem{claim_add_gen}{Claim}

\title{{\Large Tilting modules arising from two-term tilting complexes}}

\author{Hiroki Abe}
\date{}

\begin{document}

\maketitle


\begin{abstract}
We show that every two-term tilting complex over an Artin algebra has a tilting module over a certain factor algebra as a homology group. Also, we determine the endomorphism algebra of such a homology group, which is given as a certain factor algebra of the endomorphism algebra of the two-term tilting complex. Thus, every derived equivalence between Artin 
algebras given by a two-term tilting complex induces a derived equivalence between the corresponding factor algebras.
\end{abstract}

\section{Introduction}

In the representation theory of Artin algebras, the connection between tilting modules and torsion theories has been well studied. Brenner and Butler introduced the notion of tilting modules and showed that tilting modules induce torsion theories for module categories (\cite{BB}). Conversely, several authors asked when torsion theories determine tilting modules. Hoshino gave a construction of tilting modules from torsion theories for under certain conditions (\cite{Ho}). Smal{\o} characterized torsion theories which determine tilting modules using the notion of covariantly finite subcategories and contravariantly finite subcategories (\cite{Sm}). On the other hand, Rickard introduced the notion of tilting complexes as a generalization of tilting modules and showed that tilting complexes induce equivalences between derived categories of module categories, which are called derived equivalences (\cite{Ri}). Then Hoshino, Kato, and Miyachi pointed out that two-term tilting complexes induce torsion theories for module categories and studied the connection between two-term tilting complexes and torsion theories (\cite{HKM}). In this note, we show  that the torsion theories introduced by Hoshino, Kato, and Miyachi determine tilting modules. \\
\indent
Let $A$ be an Artin algebra and $T^{\bullet}$ a two-term tilting complex of $A$. We prove that the $0$-th homology group $\mathrm{H}^{0}(T^{\bullet})$ is a tilting module of $A/\mathfrak{a}$, where $\mathfrak{a}$ is the annihilator of $\mathrm{H}^{0}(T^{\bullet})$ (Theorem \ref{tilting_module}). Furthermore, we determine the endomorphism algebra of $\mathrm{H}^{0}(T^{\bullet})$. Let $B$ be the endomorphism algebra of $T^{\bullet}$. Then the endomorphism algebra of $\mathrm{H}^{0}(T^{\bullet})$ is given as $B/\mathfrak{b}$, where $\mathfrak{b}$ is the annihilator of $\mathrm{H}^{0}(T^{\bullet})$ (Theorem \ref{endring}). Thus, we know that any derived equivalence given by arbitrary two-term tilting complex always induces a derived equivalence between the corresponding factor algebras. \\ 
\indent
Throughout this note, $R$ is a commutative Artinian local ring and $A$ is an Artin $R$-algebra, i.e., $A$ is a ring endowed with a ring homomorphism $R \to A$ whose image is contained in the center of $A$ and $A$ is finitely generated as a $R$-module. We always assume that $A$ is connected, basic, and not simple. We denote by $\mathrm{mod}\text{-}A$ the category of finitely generated right $A$-modules and by $\mathcal{P}_{A}$ (resp., $\mathcal{I}_{A}$) the full subcategory of $\mathrm{mod}\text{-}A$ consisting of projective (resp., injective) modules. We denote by $A^\mathrm{op}$ the opposite ring of $A$ and consider left $A$-modules as right $A^\mathrm{op}$-modules. Sometimes, we use the notation $X_{A}$ (resp., $_{A}X$) to stress that the module $X$ considered is a right (resp., left) $A$-module. Let $X \in \mathrm{mod}\text{-}A$. We denote by $\mathrm{gen}(X)$ (resp., $\mathrm{cog}(X)$) the full subcategory of $\mathrm{mod}\text{-}A$ whose objects are generated (resp., cogenerated) by $X$. We denote by $\mathrm{add}(X)$ the full subcategory of $\mathrm{mod}\text{-}A$ whose objects are direct summands of finite direct sums of copies of $X$ and by $X^{(n)}$ the direct sum of $n$ copies of $X$. We denote by $\mathscr{K}(\mathrm{mod}\text{-}A)$, for short $\mathscr{K}(A)$, the homotopy category of cochain complexes over $\mathrm{mod}\text{-}A$ and by $\mathscr{K}^{\mathrm{b}}(\mathcal{P}_{A})$ the full triangulated subcategory of $\mathscr{K}(\mathrm{mod}\text{-}A)$ consisting of bounded complexes over $\mathcal{P}_{A}$. We denote by $\mathscr{D}(\mathrm{mod}\text{-}A)$, for short $\mathscr{D}(A)$, the derived category of cochain complexes over $\mathrm{mod}\text{-}A$ and by $\mathscr{D}^{\mathrm{b}}(\mathrm{mod}\text{-}A)$ the full triangulated subcategory of $\mathscr{D}(\mathrm{mod}\text{-}A)$ consisting of complexes which have bounded homology. We consider modules as complexes concentrated in degree zero. \\
\indent
We set $D=\mathrm{Hom}_{R}(-,E(R/\mathfrak{m}))$, where $\mathfrak{m}$ is the maximal ideal of $R$ and $E(R/\mathfrak{m})$ is an injective envelope of $R/\mathfrak{m}$, and set $\nu = DA \otimes_{A} -$, which is called the Nakayama functor of $A$. The Nakayama functor $\nu : \mathrm{mod}\text{-}A \to \mathrm{mod}\text{-}A$ induces an equivalence $\mathcal{P}_{A} \overset{\sim}{\to} \mathcal{I}_{A}$. We denote by $\nu^{-1}=\mathrm{Hom}_{A}(DA,-)$ the quasi-inverse of  $\nu$. Let $X \in \mathrm{mod}\text{-}A$, and let $P^{-1} \overset{f}{\to} P^{0} \to X \to 0$ be a minimal projective presentation. We set ${\tau}X = \mathrm{Ker}\ {\nu}(f)$, which is called the Auslander--Reiten translation. Then $\tau$ induces an equivalence between the projectively stable category of $\mathrm{mod}\text{-}A$ and the injectivitely stable category of $\mathrm{mod}\text{-}A$. We denote by $\tau^{-1}$ the quasi-inverse of $\tau$. \\
\indent
We refer to \cite{HR} for the definition and basic properties of tilting modules, to \cite{Har} and \cite{Ve} for basic results in the theory of derived categories, and to \cite{Ri} for definitions and basic properties of tilting complexes and derived equivalences. \\
\indent
The author would like to thank M. Hoshino for his helpful advice. 

\section{Preliminaries}

In this section, we recall some results on stable torsion theories given by Hoshino, Kato, and Miyachi (\cite{HKM}). We need the relationship between stable torsion theories and two-term tilting complexes of Artin algebras. 

\begin{definition}[{\cite{Di}}]\label{def_stable_torsion_theory}
A pair $(\mathcal{T},\mathcal{F})$ of full subcategories $\mathcal{T}$, $\mathcal{F}$ in $\mathrm{mod}\text{-}A$ is said to be {\it a torsion theory} for $\mathrm{mod}\text{-}A$ if the following conditions are satisfied:
\begin{enumerate}
\item[(1)] $\mathcal{T} \cap \mathcal{F} = \{ 0 \}$;  
\item[(2)] $\mathcal{T}$ is closed under factor modules; 
\item[(3)] $\mathcal{F}$ is closed under submodules; and
\item[(4)] for any $X \in \mathcal{A}$, there exists an exact sequence $0 \to X^{\prime} \to X \to X^{\prime \prime} \to 0$ with $X^{\prime} \in \mathcal{T}$ and $X^{\prime \prime} \in \mathcal{F}$.
\end{enumerate}
In particular, $\mathcal{T}$ (resp., $\mathcal{F}$) is said to be a torsion (resp., torsion-free) class. Furthermore, if $\mathcal{T}$ is stable under the Nakayama functor $\nu$, then $(\mathcal{T},\mathcal{F})$ is said to be {\it a stable torsion theory} for $\mathrm{mod}\text{-}A$. 
\end{definition}

\begin{remark}\label{rem_torsion_theory}
Let $(\mathcal{T}, \mathcal{F})$ be a torsion theory for $\mathrm{mod}\text{-}A$. 
\begin{enumerate}
\item[(1)] $\mathcal{T}$ and $\mathcal{F}$ are closed under extensions. 
\item[(2)] $(\mathcal{T}, \mathcal{F})$ is a stable torsion theory if and only if $\mathcal{F}$ is stable under $\nu^{-1}$. 
\end{enumerate}
\end{remark}

Let $T^{\bullet} \in \mathscr{K}^{\mathrm{b}}(\mathcal{P}_{A})$ be a two-term complex: 
\[
T^{\bullet} : \cdots \to 0 \to T^{-1} \overset{\alpha}{\to} T^{0} \to 0 \to \cdots .          
\]
We set the following subcategories in $\mathrm{mod}\text{-}A$: 
\[
\mathcal{T}(T^{\bullet}) = \mathrm{Ker}\ \mathrm{Hom}_{\mathscr{K}(A)}(T^{\bullet}[-1], - ), \ 
\mathcal{F}(T^{\bullet}) = \mathrm{Ker}\ \mathrm{Hom}_{\mathscr{K}(A)}(T^{\bullet}, - ). 
\]

\begin{proposition}[{\cite[Proposition 5.5]{HKM}}]\label{tilting_stable_torsion_theory}
The following are equivalent. 
\begin{enumerate}
\item[(1)] $T^{\bullet}$ is a tilting complex. 
\item[(2)] $(\mathcal{T}(T^{\bullet}),\mathcal{F}(T^{\bullet}))$ is a stable torsion theory for $\mathrm{mod}\text{-}A$. 
\end{enumerate}
\end{proposition}

\begin{definition}\label{def_Ext-proj-inj}
Let $\mathcal{C}$ be a full subcategory of $\mathrm{mod}\text{-}A$ closed under extensions. Then $M \in \mathcal{C}$ is said to be Ext-projective (resp., Ext-injective) in $\mathcal{C}$ if $\mathrm{Ext}^{1}_{A}(M, \mathcal{C}) = 0$ (resp., $\mathrm{Ext}^{1}_{A}(\mathcal{C}, M) = 0$). 
\end{definition}

\begin{remark}\label{condition_Ext-proj-inj}
Let $(\mathcal{T}, \mathcal{F})$ be a torsion theory for $\mathrm{mod}\text{-}A$. 
\begin{enumerate}
\item[(1)] For $M \in \mathcal{T}$ which is indecomposable, $M$ is Ext-projective in $\mathcal{T}$ if and only if $\tau M \in \mathcal{F}$. 
\item[(2)] For $N \in \mathcal{F}$ which is indecomposable, $N$ is Ext-injective in $\mathcal{F}$ if and only if $\tau^{-1} N \in \mathcal{F}$. 
\end{enumerate}
\end{remark}

\begin{proposition}[{\cite[Proposition 5.7]{HKM}}]\label{Ext-proj-gen}
Assume that $T^{\bullet}$ is a tilting complex. Then the following hold. 
\begin{enumerate}
\item[(1)] $\mathcal{T}(T^{\bullet})=\mathrm{gen}(\mathrm{H}^{0}(T^{\bullet}))$ and $\mathrm{H}^{0}(T^{\bullet})$ is Ext-projective in $\mathcal{T}(T^{\bullet})$. 
\item[(2)] $\mathcal{F}(T^{\bullet})=\mathrm{cog}(\mathrm{H}^{-1}(\nu T^{\bullet}))$ and $\mathrm{H}^{-1}(\nu T^{\bullet})$ is Ext-injective in $\mathcal{F}(T^{\bullet})$.   
\end{enumerate}
\end{proposition}

\begin{theorem}[{\cite[Theorem 5.8]{HKM}}]\label{const_two_term_tilting}
Let $(\mathcal{T}, \mathcal{F})$ be a stable torsion theory for $\mathrm{mod}\text{-}A$. Assume that there exist $X \in \mathcal{T}$ and $Y \in \mathcal{F}$ satisfying the following conditions: 
\begin{enumerate}
\item[(1)] $\mathcal{T}=\mathrm{gen}(X)$ and $X$ is Ext-projective in $\mathcal{T}$; and 
\item[(2)] $\mathcal{F}=\mathrm{cog}(Y)$ and $Y$ is Ext-injective in $\mathcal{F}$. 
\end{enumerate}
Let $P^{\bullet}_{X}$ be a minimal projective presentation of $X$ and $I^{\bullet}_{Y}$ be a minimal injective presentation of $Y$, and set $T^{\bullet}_{X,Y} = P^{\bullet}_{X} \oplus \nu^{-1}I^{\bullet}_{Y}[1]$. Then $T^{\bullet}_{X,Y} \in \mathscr{K}^{\mathrm{b}}(\mathcal{P}_{A})$ is a tilting complex such that $\mathcal{T}=\mathcal{T}(T^{\bullet}_{X,Y})$ and $\mathcal{F}=\mathcal{F}(T^{\bullet}_{X,Y})$. 
\end{theorem}

\section{Tilting modules arising from two-term tilting complexes}

For $X \in \mathrm{mod}\text{-}A$, we use the notation $\mathrm{gen}(X_{A})$ (resp., $\mathrm{cog}(X_{A}), \mathrm{add}(X_{A}))$ to stress that it is considered as a subcategory of $\mathrm{mod}\text{-}A$. We denote by $\mathrm{ann}_{A}(X)$ the annihilator of $X$. 

\begin{lemma}\label{partial_tilting}
Assume that $X \in \mathrm{mod}\text{-}A$ is Ext-projective in $\mathrm{gen}(X_{A})$, and set $\mathfrak{a}=\mathrm{ann}_{A}(X)$. Then the following hold. 
\begin{enumerate}
\item[(1)] $\mathrm{proj\ dim}\ X_{A/\mathfrak{a}} \le 1$. 
\item[(2)] $\mathrm{Ext}^{1}_{A/\mathfrak{a}}(X,X)=0$. 
\item[(3)] There exists an exact sequence $0 \to A/\mathfrak{a} \to X^{0} \to X^{1} \to 0$ in $\mathrm{mod}\text{-}A/\mathfrak{a}$ such that $X^{0} \in \mathrm{add}(X_{A/\mathfrak{a}})$ and $X^{1} \in \mathrm{gen}(X_{A/\mathfrak{a}})$ which is Ext-projective in $\mathrm{gen}(X_{A/\mathfrak{a}})$.  
\end{enumerate}
\end{lemma}
\begin{proof}
Note first that the canonical full embedding $\mathrm{mod}\text{-}A/\mathfrak{a} \hookrightarrow \mathrm{mod}\text{-}A$ induces $\mathrm{gen}(X_{A/\mathfrak{a}}) = \mathrm{gen}(X_{A})$. \\
\indent
(1) Since $X_{A/\mathfrak{a}}$ is Ext-projective in $\mathrm{gen}(X_{A/\mathfrak{a}})$ by assumption, the pair $(\mathrm{gen}(X_{A/\mathfrak{a}}), \mathrm{Ker}\ \mathrm{Hom}_{A/\mathfrak{a}}(X,-))$ is a torsion theory for $\mathrm{mod}\text{-}A/\mathfrak{a}$. Since $DX$ is faithful as a left $A/\mathfrak{a}$-module, we have $D(A/\mathfrak{a}) \in \mathrm{gen}(X_{A/\mathfrak{a}})$. Let $Z$ be an indecomposable direct summand of $X_{A/\mathfrak{a}}$. We may assume that $Z$ is not projective in $\mathrm{mod}\text{-}A/\mathfrak{a}$. Since $Z$ is Ext-projective in $\mathrm{gen}(X_{A/\mathfrak{a}})$, we have ${\tau}Z_{A/\mathfrak{a}} \in \mathrm{Ker}\ \mathrm{Hom}_{A/\mathfrak{a}}(X,-)$. Let $0 \to \tau Z \to I^{0} \to I^{1}$ be a minimal injective presentation in $\mathrm{mod}\text{-}A/\mathfrak{a}$. Then we have an exact sequence 
\[
0 \to \nu^{-1}(\tau Z) \to \nu^{-1}I^{0} \to \nu^{-1}I^{1} \to Z \to 0. 
\] 
Since $\nu^{-1}I^{0}, \nu^{-1}I^{1} \in \mathcal{P}_{A/\mathfrak{a}}$ and $\nu^{-1}(\tau Z) = \mathrm{Hom}_{A/\mathfrak{a}}(D(A/\mathfrak{a}), \tau Z) =0$, the above exact sequence gives a minimal projective resolution of $Z_{A/\mathfrak{a}}$. Thus, we have $\mathrm{proj\ dim}\ X_{A/\mathfrak{a}} \le 1$. \\
\indent 
(2) It follows by the assumption that $X_{A/\mathfrak{a}}$ is Ext-projective in $\mathrm{gen}(X_{A/\mathfrak{a}})$. The assertion follows. \\
\indent
(3) Since $X$ is faithful as a right $A/\mathfrak{a}$-module, there exist generators
\[
f_{1}, \cdots, f_{d} \in \mathrm{Hom}_{A/\mathfrak{a}}(A/\mathfrak{a}, X)
\]
as a left $\mathrm{End}_{A/\mathfrak{a}}(X)$-module such that 
\[
f=\left[ \begin{array}{c} f_{1} \\ \vdots \\ f_{d} \end{array} \right] : A/\mathfrak{a} \to X^{(d)}, a \mapsto \left[ \begin{array}{c} f_{1}(a) \\ \vdots \\ f_{d}(a) \end{array} \right] 
\]
is monic. We show that $\mathrm{Cok}\ f$ is Ext-projective in $\mathrm{gen}(X_{A/\mathfrak{a}})$. Let $N \in \mathrm{gen}(X_{A/\mathfrak{a}})$. Then there exists an epimorphism $\varepsilon: X^{(n)} \to N$, and we have a commutative diagram 
\[
\begin{CD}
\mathrm{Hom}_{A/\mathfrak{a}}(X^{(d)}, X^{(n)}) @>{\mathrm{Hom}_{A/\mathfrak{a}}(X^{(d)}, \varepsilon)}>> \mathrm{Hom}_{A/\mathfrak{a}}(X^{(d)}, N) \\
@V{\mathrm{Hom}_{A/\mathfrak{a}}(f, X)}VV  @VV{\mathrm{Hom}_{A/\mathfrak{a}}(f, N)}V   \\
\mathrm{Hom}_{A/\mathfrak{a}}(A/\mathfrak{a}, X^{(n)}) @>>{\mathrm{Hom}_{A/\mathfrak{a}}(A/\mathfrak{a}, \varepsilon)}> \mathrm{Hom}_{A/\mathfrak{a}}(A/\mathfrak{a}, N).             
\end{CD}
\]
Since  $\mathrm{Hom}_{A/\mathfrak{a}}(A/\mathfrak{a}, \varepsilon)$ is epic and $\mathrm{Hom}_{A/\mathfrak{a}}(f, X)$ is also epic by the construction, we have ${\mathrm{Hom}_{A/\mathfrak{a}}(f, N)}$ is epic and hence  $\mathrm{Ext}^{1}_{A/\mathfrak{a}}(\mathrm{Cok}\ f, N)=0$. Thus, $\mathrm{Cok}\ f$ is Ext-projective in $\mathrm{gen}(X_{A/\mathfrak{a}})$. 
\end{proof}

\begin{lemma}\label{partial_cotilting}
Assume that $Y \in \mathrm{mod}\text{-}A$ is Ext-injective in $\mathrm{cog}(Y_{A})$, and set $\mathfrak{a}^{\prime}=\mathrm{ann}_{A}(Y)$. Then the following hold. 
\begin{enumerate}
\item[(1)] $\mathrm{inj\ dim}\ Y_{A/\mathfrak{a}^{\prime}} \le 1$. 
\item[(2)] $\mathrm{Ext}^{1}_{A/\mathfrak{a}^{\prime}}(Y,Y)=0$. 
\item[(3)] There exists an exact sequence $0 \to Y^{1} \to Y^{0} \to A/\mathfrak{a}^{\prime} \to 0$ in $\mathrm{mod}\text{-}A/\mathfrak{a}^{\prime}$ such that $Y^{0} \in \mathrm{add}(Y_{A/\mathfrak{a}^{\prime}})$ and $Y^{1} \in \mathrm{cog}(Y_{A/\mathfrak{a}^{\prime}})$ which is Ext-injective in $\mathrm{cog}(Y_{A/\mathfrak{a}^{\prime}})$.  
\end{enumerate}
\end{lemma}
\begin{proof}
There exists an equivalence $D(\mathrm{cog}(Y)) \cong \mathrm{gen}(DY)$ as subcategories in $\mathrm{mod}\text{-}A^{\mathrm{op}}$, and hence $DY \in \mathrm{mod}\text{-}A^{\mathrm{op}}$ is Ext-projective in $\mathrm{gen}(DY)$. The assertion follows by Lemma \ref{partial_tilting}.  
\end{proof}

Throughout the rest of this section, let $T^{\bullet} \in \mathscr{K}^{\mathrm{b}}(\mathcal{P}_{A})$ be a two-term tilting complex: 
\[
T^{\bullet} : \cdots \to 0 \to T^{-1} \overset{\alpha}{\to} T^{0} \to 0 \to \cdots .          
\]

\begin{lemma}\label{add_gen}
For any $M, N \in \mathrm{mod}\text{-}A$, the following hold. 
\begin{enumerate}
\item[(1)] $M \in \mathrm{add}(\mathrm{H}^{0}(T^{\bullet}))$ if and only if $M$ is $\mathrm{Ext}$-projective in $\mathrm{gen}(\mathrm{H}^{0}(T^{\bullet}))$. 
\item[(2)] $N \in \mathrm{add}(\mathrm{H}^{-1}({\nu}T^{\bullet}))$ if and only if $N$ is $\mathrm{Ext}$-injective in $\mathrm{cog}(\mathrm{H}^{-1}({\nu}T^{\bullet}))$.
\end{enumerate}
\end{lemma}
\begin{proof}
(1) We know from Proposition \ref{Ext-proj-gen} that $\mathrm{H}^{0}(T^{\bullet})$ is Ext-projective in $\mathrm{gen}(\mathrm{H}^{0}(T^{\bullet}))$. Let $X$ be the direct sum of all indecomposable non-projective Ext-projective modules in $\mathrm{gen}(\mathrm{H}^{0}(T^{\bullet}))$ which are not contained in $\mathrm{add}(\mathrm{H}^{0}(T^{\bullet}))$. Then $\mathrm{add}(\mathrm{H}^{0}(T^{\bullet}) \oplus X)$ coincides with the class of all Ext-projective modules in $\mathrm{gen}(\mathrm{H}^{0}(T^{\bullet}))$. On the other hand, since we have $\mathrm{gen}(\mathrm{H}^{0}(T^{\bullet})) = \mathrm{gen}(\mathrm{H}^{0}(T^{\bullet}) \oplus X)$, it follows by Propositions \ref{tilting_stable_torsion_theory} and \ref{Ext-proj-gen} that the pair  
\[
(\mathrm{gen}(\mathrm{H}^{0}(T^{\bullet}) \oplus X), \mathrm{cog}(\mathrm{H}^{-1}({\nu}T^{\bullet}))
\]
is a stable torsion theory in $\mathrm{mod}\text{-}A$. Let $P^{\bullet}$ be the minimal projective presentation of $\mathrm{H}^{0}(T^{\bullet}) \oplus X$ and $I^{\bullet}$ be the minimal injective presentation of $\mathrm{H}^{-1}({\nu}T^{\bullet})$, and set $U^{\bullet}=P^{\bullet} \oplus {\nu^{-1}}I^{\bullet}[1]$. Then $U^{\bullet}$ is a tilting complex by Theorem \ref{const_two_term_tilting}. 
\begin{claim_add_gen}
$\mathrm{add}(\mathrm{H}^{0}(T^{\bullet})) \subset \mathrm{add}(\mathrm{H}^{0}(T^{\bullet}) \oplus X) \subset \mathrm{add}(\mathrm{H}^{0}(U^{\bullet}))$.
\end{claim_add_gen}
\begin{proof}
The first inclusion is obvious. Since $\mathrm{H}^{0}(U^{\bullet}) = \mathrm{H}^{0}(P^{\bullet}) \oplus \mathrm{H}^{1}({\nu^{-1}}I^{\bullet}) \cong \mathrm{H}^{0}(T^{\bullet}) \oplus X \oplus \mathrm{H}^{1}({\nu^{-1}}I^{\bullet})$, the second inclusion follows. 
\end{proof}
\begin{claim_add_gen}
$\mathrm{add}(T^{\bullet}) = \mathrm{add}(U^{\bullet})$.
\end{claim_add_gen}
\begin{proof}
Set $W^{\bullet}=T^{\bullet} \oplus U^{\bullet}$. Then we have $\mathrm{gen}(\mathrm{H}^{0}(W^{\bullet})) = \mathrm{gen}(\mathrm{H}^{0}(U^{\bullet}))$ because $\mathrm{gen}(\mathrm{H}^{0}(T^{\bullet})) \subset \mathrm{gen}(\mathrm{H}^{0}(U^{\bullet}))$. Similarly, since $\mathrm{cog}(\mathrm{H}^{-1}({\nu}T^{\bullet})) \subset \mathrm{cog}(\mathrm{H}^{-1}({\nu}U^{\bullet}))$, we have $\mathrm{cog}(\mathrm{H}^{-1}({\nu}W^{\bullet})) = \mathrm{cog}(\mathrm{H}^{-1}({\nu}U^{\bullet}))$. It then follows by Proposition \ref{tilting_stable_torsion_theory} that the pair 
\[
(\mathrm{gen}(\mathrm{H}^{0}(W^{\bullet})), \mathrm{cog}(\mathrm{H}^{-1}({\nu}W^{\bullet})))
\]
is a stable torsion theory for $\mathrm{mod}\text{-}A$, and hence $W^{\bullet} \in \mathscr{K}^{\mathrm{b}}(\mathcal{P}_{A})$ is a tilting complex. We set $\Lambda = \mathrm{End}_{\mathscr{D}(A)}(W^{\bullet})$. We denote by $e \in \Lambda$ the idempotent corresponding to $T^{\bullet}$ and by $f \in \Lambda$ the idempotent corresponding to $U^{\bullet}$, i.e., $e{\Lambda}e \cong \mathrm{End}_{\mathscr{D}(A)}(T^{\bullet})$ and $f{\Lambda}f \cong \mathrm{End}_{\mathscr{D}(A)}(U^{\bullet})$. Since $\Lambda_{\Lambda} \cong e{\Lambda} \oplus f{\Lambda}$, the derived equivalence 
\[
F : \mathscr{D}^{\mathrm{b}}(\mathrm{mod}\text{-}\Lambda) \overset{\sim}{\to} \mathscr{D}^{\mathrm{b}}(\mathrm{mod}\text{-}A), \Lambda_{\Lambda} \mapsto W^{\bullet}
\]
induces 
\[
\mathscr{D}^{\mathrm{b}}(\mathrm{mod}\text{-}\Lambda) \overset{\sim}{\to} \mathscr{D}^{\mathrm{b}}(\mathrm{mod}\text{-}e{\Lambda}e), e{\Lambda} \mapsto e{\Lambda}e.
\]
Thus, $e{\Lambda} \in \mathcal{P}_{\Lambda}$ is a projective generator in $\mathrm{mod}\text{-}\Lambda$, i.e., $\Lambda \in \mathrm{add}(e{\Lambda})$. Applying the quasi-inverse of $F$, we have $W^{\bullet} \in \mathrm{add}(T^{\bullet})$ by the additivity of $F$. Similarly, $\Lambda \in \mathrm{add}(f{\Lambda})$ and hence $W^{\bullet} \in \mathrm{add}(U^{\bullet})$. It follows that $\mathrm{add}(T^{\bullet}) = \mathrm{add}(W^{\bullet}) = \mathrm{add}(U^{\bullet})$. 
\end{proof}
By the above claims, we have $\mathrm{add}(\mathrm{H}^{0}(T^{\bullet})) = \mathrm{add}(\mathrm{H}^{0}(T^{\bullet}) \oplus X)$. The assertion follows. \\
\indent
(2) Note first that $\mathrm{Hom}^{\bullet}_{A}(T^{\bullet},A) \in \mathscr{K}^{\mathrm{b}}(\mathcal{P}_{A^{\mathrm{op}}})$ is a two-term tilting complex, where $\mathrm{Hom}^{\bullet}(-,-)$ denotes the single complex associated with the double hom complex. We know from (1) that  $M \in \mathrm{add}(\mathrm{H}^{1}(\mathrm{Hom}^{\bullet}_{A}(T^{\bullet},A)))$ if and only if $M$ is Ext-projective in $\mathrm{gen}(\mathrm{H}^{1}(\mathrm{Hom}^{\bullet}_{A}(T^{\bullet},A)))$. Since $D(\mathrm{gen}(\mathrm{H}^{1}(\mathrm{Hom}^{\bullet}_{A}(T^{\bullet},A))) \cong \mathrm{cog}(\mathrm{H}^{-1}({\nu}T^{\bullet}))$, it follows that $M$ is Ext-projective in $\mathrm{gen}(\mathrm{H}^{1}(\mathrm{Hom}^{\bullet}_{A}(T^{\bullet},A)))$ if and only if $DM$ is Ext-injective in $\mathrm{cog}(\mathrm{H}^{-1}({\nu}T^{\bullet}))$. Also, since there exists an equivalence $D(\mathrm{add}(\mathrm{H}^{1}(\mathrm{Hom}^{\bullet}_{A}(T^{\bullet},A))) \cong \mathrm{add}(\mathrm{H}^{-1}({\nu}T^{\bullet}))$, the assertion follows. 
\end{proof}

The next theorem is a direct consequence of the previous three lemmas. We set $\mathfrak{a} = \mathrm{ann}_{A}(\mathrm{H}^{0}(T^{\bullet}))$ and $\mathfrak{a}^{\prime} = \mathrm{ann}_{A}(\mathrm{H}^{-1}({\nu}T^{\bullet}))$.  

\begin{theorem}\label{tilting_module}
The following hold. 
\begin{enumerate}
\item[(1)] $\mathrm{H}^{0}(T^{\bullet})$ is a tilting module in $\mathrm{mod}\text{-}A/\mathfrak{a}$. 
\item[(2)] $\mathrm{H}^{-1}({\nu}T^{\bullet})$ is a cotilting module in $\mathrm{mod}\text{-}A/\mathfrak{a}^{\prime}$, i.e., $D(\mathrm{H}^{-1}({\nu}T^{\bullet}))$ is a tilting module in $\mathrm{mod}\text{-}{(A/\mathfrak{a}^{\prime})}^{\mathrm{op}}$. 
\end{enumerate}
\end{theorem}

\begin{remark}\label{covariant_finite}
From \cite{Sm} and the above theorem, we know that $\mathcal{T}(T^{\bullet})$ (resp., $\mathcal{F}(T^{\bullet}))$ is covariantly (resp., contravariantly) finite subcategory of $\mathrm{mod}\text{-}A$. 
\end{remark}

As the final result, we determine the endomorphism algebras of $\mathrm{H}^{0}(T^{\bullet})$ and $\mathrm{H}^{-1}({\nu}T^{\bullet})$. It is easy to see that the homomorphism
\[
\mathrm{H}^{0}(-) : \mathrm{End}_{\mathscr{K}(A)}(T^{\bullet}) \to \mathrm{End}_{A/\mathfrak{a}}(\mathrm{H}^{0}(T^{\bullet})), \varphi \mapsto \mathrm{H}^{0}(\varphi)
\]
is a surjective algebra homomorphism. Thus, we need only to calculate the kernel of the above algebra homomorphism. In order to do this, we deal with $\mathrm{Hom}_{\mathscr{K}(A)}(A,T^{\bullet})$ instead of $\mathrm{H}^{0}(T^{\bullet})$. This is justified by the fact that there exists an isomorphism $\mathrm{H}^{0}(T^{\bullet}) \cong \mathrm{Hom}_{\mathscr{K}(A)}(A,T^{\bullet})$ as right $A$-modules. Similarly, we may deal with $\mathrm{Hom}_{\mathscr{K}(A)}(A,{\nu}T^{\bullet}[-1])$ instead of $\mathrm{H}^{-1}({\nu}T^{\bullet})$. We set $B= \mathrm{End}_{\mathscr{K}(A)}(T^{\bullet})$ and set
\[
\mathfrak{b} = \mathrm{ann}_{B}(\mathrm{Hom}_{\mathscr{K}(A)}(A,T^{\bullet})), \ 
\mathfrak{b}^{\prime} = \mathrm{ann}_{B}(\mathrm{Hom}_{\mathscr{K}(A)}(A,{\nu}T^{\bullet}[-1])).
\]

\begin{theorem}\label{endring}
We have the following algebra isomorphisms. 
\begin{enumerate}
\item[(1)] $\mathrm{End}_{A/\mathfrak{a}}(\mathrm{Hom}_{\mathscr{K}(A)}(A,T^{\bullet})) \cong B/\mathfrak{b}$. 
\item[(2)] $\mathrm{End}_{A/\mathfrak{a}^{\prime}}(\mathrm{Hom}_{\mathscr{K}(A)}(A,{\nu}T^{\bullet}[-1])) \cong B/\mathfrak{b}^{\prime}$. 
\end{enumerate}
\end{theorem}
\begin{proof}
(1) Since there exists a surjective algebra homomorphism 
\[
\theta : B \to \mathrm{End}_{A/\mathfrak{a}}(\mathrm{Hom}_{\mathscr{K}(A)}(A,T^{\bullet})),
\]
which is induced by the functor $\mathrm{H}^{0}(-)$, we have an algebra isomorphism 
\[
\mathrm{End}_{A/\mathfrak{a}}(\mathrm{Hom}_{\mathscr{K}(A)}(A,T^{\bullet})) \cong B/\mathrm{Ker}\ \theta. 
\]
We will show that $\mathrm{Ker}\ \theta = \mathfrak{b}$. Let $\varphi \in B$: 
\[
\begin{CD}
T^{\bullet} @. : \cdots @>>> 0 @>>> T^{-1} @>{\alpha}>> T^{0} @>>> 0 @>>> \cdots \\
@V{\varphi}VV. @. @VVV  @V{{\varphi}^{-1}}VV @VV{{\varphi}^{0}}V @VVV @.  \\
T^{\bullet} @. : \cdots @>>> 0 @>>> T^{-1} @>>{\alpha}> T^{0} @>>> 0 @>>> \cdots .             
\end{CD}
\]
Then we have a commutative diagram with exact rows 
\[
\begin{CD}
T^{-1} @>{\alpha}>> T^{0} @>{\varepsilon}>> \mathrm{Cok}\ \alpha @>>> 0 \\
@V{{\varphi}^{-1}}VV @VV{{\varphi}^{0}}V @VV{\theta(\varphi)}V @.  \\
T^{-1} @>>{\alpha}> T^{0} @>>{\varepsilon}> \mathrm{Cok}\ \alpha @>>> 0.           
\end{CD}
\]
We assume first that $\theta(\varphi) \ne 0$. Then there exists $t \in T^{0}$ such that $(\theta(\varphi) \circ \varepsilon)(t) \ne 0$. We define $\psi : A \to T^{0}, 1_{A} \mapsto t$. Then $(\varphi^{0} \circ \psi)(1_{A}) \notin \mathrm{Ker}\ \varepsilon = \mathrm{Im}\ \alpha$. Since $(\alpha \circ h)(1_{A}) \in \mathrm{Im}\ \alpha$ for all $h \in \mathrm{Hom}_{A}(A, T^{-1})$. Therefore, $\varphi \circ \psi$ is not homotopic to zero and hence $\varphi \notin \mathfrak{b}$. Thus, we have $\mathfrak{b} \subset \mathrm{Ker}\ \theta$. Conversely, we assume that $\theta(\varphi) = 0$. Since $\varphi^{0}$ factors through $\mathrm{Im}\ \alpha$, there exists $h \in \mathrm{Hom}_{A}(T^{0}, T^{-1})$ such that $\varphi^{0}=\alpha \circ h$ by the projectivity of $T^{0}$. Thus, for any $\sigma \in \mathrm{Hom}_{\mathscr{K}(A)}(A,T^{\bullet})$, we have $\varphi^{0} \circ \sigma = \alpha \circ h \circ \sigma$. Therefore, $\varphi \circ \sigma$ is homotopic to zero and hence $\varphi \in \mathfrak{b}$. This shows that $\mathrm{Ker}\ \theta \subset  \mathfrak{b}$. \\
\indent
(2) Since $\nu : \mathscr{K}^{\mathrm{b}}(\mathcal{P}_{A}) \overset{\sim}{\to} \mathscr{K}^{\mathrm{b}}(\mathcal{I}_{A})$, we have $B \cong \mathrm{End}_{\mathscr{K}(A)}({\nu}T^{\bullet})$ as algebras. It is easy to see that there exists a surjective algebra homomorphism 
\[
\theta^{\prime} : B \to \mathrm{End}_{A/\mathfrak{a^{\prime}}}(\mathrm{Hom}_{\mathscr{K}(A)}(A,{\nu}T^{\bullet}[-1])),
\]
which is induced by the functor $\mathrm{H}^{-1}(-)$. We will show that $\mathrm{Ker}\ \theta^{\prime} = \mathfrak{b}^{\prime}$. Let $\phi \in \mathrm{End}_{\mathscr{K}(A)}({\nu}T^{\bullet})$: 
\[
\begin{CD}
{\nu}T^{\bullet} @. : \cdots @>>> 0 @>>> {\nu}T^{-1} @>{\beta}>> {\nu}T^{0} @>>> 0 @>>> \cdots \\
@V{\phi}VV. @. @VVV  @V{{{\phi}^{-1}}}VV @VV{{{\phi}^{0}}}V @VVV @.  \\
{\nu}T^{\bullet} @. : \cdots @>>> 0 @>>> {\nu}T^{-1} @>>{\beta}> {\nu}T^{0} @>>> 0 @>>> \cdots,          
\end{CD}
\]
where $\beta = \nu(\alpha)$. Then we have a commutative diagram with exact rows 
\[
\begin{CD}
0 @>>> \mathrm{Ker}\ \beta @>{\iota}>> {\nu}T^{-1} @>{\beta}>> {\nu}T^{0} \\
@. @V{\theta^{\prime}(\phi)}VV @V{{\phi}^{-1}}VV @VV{{\phi}^{0}}V \\
0 @>>> \mathrm{Ker}\ \beta @>>{\iota}> {\nu}T^{-1} @>>{\beta}> {\nu}T^{0}. 
\end{CD}
\]
We assume first that $\theta^{\prime}(\phi) \ne 0$. Then there exists $x \in \mathrm{Ker}\ \beta$ such that $(\phi^{-1} \circ \iota)(x) = (\iota \circ \theta^{\prime}(\phi))(x) \ne 0$. We set $t = \iota(x) \in {\nu}T^{-1}$ and define $\eta : A \to {\nu}T^{-1}, 1_{A} \mapsto t$. Then, since $\eta$ satisfies $\beta \circ \eta = 0$ and $\phi^{-1} \circ \eta \ne 0$, we have $\eta \in \mathrm{Hom}_{\mathscr{K}(A)}(A,{\nu}T^{\bullet}[-1])$ and $\eta$ is not homotopic to zero. Thus, $\mathfrak{b^{\prime}} \subset \mathrm{Ker}\ \theta^{\prime}$. Conversely, we assume that $\theta^{\prime}(\phi) = 0$. For any $\rho \in \mathrm{Hom}_{\mathscr{K}(A)}(A,{\nu}T^{\bullet}[-1])$, we have $\phi^{-1} \circ \rho =0$ because $\rho$ factors through $\mathrm{Ker}\ \beta$. Thus, we have $\phi \in \mathfrak{b}^{\prime}$ and hence $\mathrm{Ker}\ \theta^{\prime} \subset  \mathfrak{b}^{\prime}$.   
\end{proof}

\section{Example}

In this section, we demonstrate our results through an example. Let $A$ be the path algebra defined by the quiver 
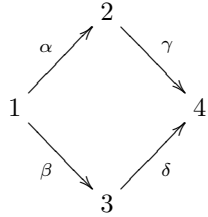
\[
\xymatrix{
 & 2 \ar [rd]^{\gamma} & \\
1 \ar [ru]^{\alpha} \ar [rd]_{\beta} & & 4 \\
 & 3 \ar [ru]_{\delta}&
}
\]
with relations $\alpha\gamma=\beta\delta=0$. We denote by $e_{i}$ the empty path corresponding to the vertex $i=1, \cdots, 4$. The Auslander--Reiten quiver of $A$ is given by the following: 
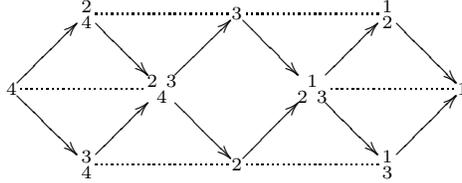
\[
\begin{xy}
(-40,10) *{\begin{smallmatrix} 2 \\ 4 \end{smallmatrix}}="P2",
(-50,0) *{\begin{smallmatrix} 4 \end{smallmatrix}}="P4",
(-40,-10) *{\begin{smallmatrix} 3 \\ 4 \end{smallmatrix}}="P3",
(-30,0) *{\begin{smallmatrix} 2 \ 3 \\ 4 \end{smallmatrix}}="I4",
(-20,10) *{\begin{smallmatrix} 3 \end{smallmatrix}}="S3",
(-20,-10) *{\begin{smallmatrix} 2 \end{smallmatrix}}="S2",
(-10,0) *{\begin{smallmatrix} 1 \\ 2 \ 3 \end{smallmatrix}}="P1",
(0,10) *{\begin{smallmatrix} 1 \\ 2 \end{smallmatrix}}="I2",
(0,-10) *{\begin{smallmatrix} 1 \\ 3 \end{smallmatrix}}="I3",
(10,0) *{\begin{smallmatrix} 1 \end{smallmatrix}}="I1",
\ar "P4" ; "P2"  
\ar "P4" ; "P3" 
\ar "P2" ; "I4"
\ar "P3" ; "I4"  
\ar@{.} "P4" ; "I4"
\ar "I4" ; "S3"  
\ar "I4" ; "S2" 
\ar "S3" ; "P1"
\ar "S2" ; "P1"  
\ar@{.} "P2" ; "S3"
\ar@{.} "P3" ; "S2"
\ar "P1" ; "I2"  
\ar "P1" ; "I3" 
\ar "I2" ; "I1"
\ar "I3" ; "I1"  
\ar@{.} "S3" ; "I2"
\ar@{.} "S2" ; "I3"
\ar@{.} "P1" ; "I1"
\end{xy}
\]
where each indecomposable module is represented by its composition factors and $\tau$-orbits are denoted by $\xymatrix{\bullet \ar@{.} [r] & \bullet}$. It is not difficult to see that the following pair gives a stable torsion theory for $\mathrm{mod}\text{-}A$: 
\[
\mathcal{T}=\{\begin{smallmatrix} 1 \\ 2 \ 3 \end{smallmatrix}, \begin{smallmatrix} 1 \\ 2 \end{smallmatrix}, \begin{smallmatrix} 1 \\ 3 \end{smallmatrix}, \begin{smallmatrix} 1 \end{smallmatrix} \}\ \text{and} \ \mathcal{F} = \{\begin{smallmatrix} 4 \end{smallmatrix}, \begin{smallmatrix} 2 \\ 4 \end{smallmatrix}, \begin{smallmatrix} 3 \\ 4 \end{smallmatrix}, \begin{smallmatrix} 2 \ 3 \\ 4 \end{smallmatrix}, \begin{smallmatrix} 3 \end{smallmatrix}, \begin{smallmatrix} 2 \end{smallmatrix} \},
\]
where $\mathcal{T}$ is a torsion class and $\mathcal{F}$ is a torsion-free class. We set 
\[
X= \begin{smallmatrix} 1\\ 2 \ 3 \end{smallmatrix}, \quad Y=\begin{smallmatrix} 2\ 3 \\ 4 \end{smallmatrix} \oplus \begin{smallmatrix} 3 \end{smallmatrix} \oplus \begin{smallmatrix} 2 \end{smallmatrix}. 
\]
Then $\mathcal{T}=\mathrm{gen}(X)$ and $X$ is Ext-projective in $\mathcal{T}$, and $\mathcal{F}= \mathrm{cog}(Y)$ and $Y$ is Ext-injective in $\mathcal{F}$. According to Theorem \ref{const_two_term_tilting}, we have a two-term tilting complex $T^{\bullet}=T^{\bullet}_{1} \oplus T^{\bullet}_{2} \oplus T^{\bullet}_{3} \oplus T^{\bullet}_{4}$, where
\[
T^{\bullet}_{1}= 0 \to \begin{smallmatrix} 1\\ 2 \ 3 \end{smallmatrix}, \quad T^{\bullet}_{2}= \begin{smallmatrix} 2 \\ 4 \end{smallmatrix} \to \begin{smallmatrix} 1\\ 2 \ 3 \end{smallmatrix}, \quad T^{\bullet}_{3}= \begin{smallmatrix} 3 \\ 4 \end{smallmatrix} \to \begin{smallmatrix} 1\\ 2 \ 3 \end{smallmatrix}, \quad T^{\bullet}_{4}= \begin{smallmatrix} 4 \end{smallmatrix} \to 0. 
\]
Thus, we have 
\[
\mathrm{H}^{0}(T^{\bullet})=\begin{smallmatrix} 1\\ 2 \ 3 \end{smallmatrix} \oplus \begin{smallmatrix} 1 \\ 3 \end{smallmatrix} \oplus \begin{smallmatrix} 1 \\ 2 \end{smallmatrix}
\]
as a right $A$-module. Since $\mathfrak{a} = \mathrm{ann}_{A}(\mathrm{H}^{0}(T^{\bullet}))$ is a two-sided ideal generated by $e_{4}, \gamma, \delta$, the factor algebra $A/\mathfrak{a}$ is defined by the quiver 
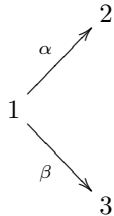
\[
\xymatrix{
 & 2  \\
1 \ar [ru]^{\alpha} \ar [rd]_{\beta} &  \\
 & 3 
}
\]
without relations. Next, it is not difficult to see that $B= \mathrm{End}_{\mathscr{K}(A)}(T^{\bullet})$ is defined by the quiver 
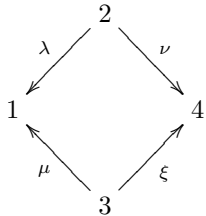
\[
\xymatrix{
 & 2 \ar [ld]_{\lambda} \ar [rd]^{\nu} & \\
1 & & 4 \\
 & 3 \ar [lu]^{\mu} \ar [ru]_{\xi}&
}
\]
without relations. Then we have 
\begin{align*}
\mathrm{Hom}_{\mathscr{K}(A)}(A,T^{\bullet}) &= \bigoplus_{i=1}^{4}\mathrm{Hom}_{\mathscr{K}(A)}(e_{i}A,T^{\bullet}) \\
&= \begin{smallmatrix} 1\\ 2 \ 3 \end{smallmatrix} \oplus \begin{smallmatrix} 1 \\ 3 \end{smallmatrix} \oplus \begin{smallmatrix} 1 \\ 2 \end{smallmatrix} \oplus 0 
\end{align*}
as a left $B$-module. Thus, $\mathfrak{b} = \mathrm{ann}_{B}(\mathrm{Hom}_{\mathscr{K}(A)}(A,T^{\bullet}))$ is a two-sided ideal generated by $\nu, \xi$ and the empty path corresponding to the vertex $4$. Therefore, the factor algebra $B/\mathfrak{b}$ is defined by the quiver 
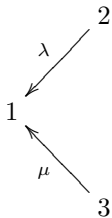
\[
\xymatrix{
 & 2 \ar [ld]_{\lambda} \\
1 & \\
 & 3 \ar [lu]^{\mu} 
}
\]
without relations. It follows by Theorems \ref{tilting_module} and \ref{endring} that $A/\mathfrak{a}$ and $B/\mathfrak{b}$ are derived equivalent to each other.

\vspace{5pt}
\begin{flushleft}
Institute of Mathematics \\
University of Tsukuba \\
Ibaraki 305-8571, JAPAN \\
{\it E-mail address}: {\rm abeh@math.tsukuba.ac.jp} 
\end{flushleft}

\end{document}